\documentclass[12pt, reqno]{amsart}
\usepackage{amsmath,amsfonts,amsbsy,amsgen,amscd,mathrsfs,amssymb,amsthm}
\usepackage[usenames,dvipsnames]{xcolor}
\usepackage[colorlinks=true,citecolor=red,linkcolor=blue]{hyperref}
\usepackage[a4paper,margin=1in]{geometry}
\usepackage{setspace}
\usepackage{amsmath}
\allowdisplaybreaks[4]
\numberwithin{equation}{section}
\newtheorem{theorem}{Theorem}[section]
\newtheorem{definition}[theorem]{Definition}
\newtheorem{lemma}[theorem]{Lemma}

\newtheorem{remark}[theorem]{Remark}

\title[the Orlicz Minkowski Problem for $q$-torsional rigidity]{Flow by Gauss Curvature to the Orlicz Minkowski Problem for $q$-torsional rigidity}

\author{Xia Zhao and Peibiao Zhao}


\begin{document}
\begin{abstract}
The celebrated Minkowski problem for the torsional rigidity ($2$-torsional rigidity) was firstly studied by Colesanti and Fimiani \cite{CA} using variational method. Moreover, Hu, Liu and Ma \cite{HJ} also studied the Minkowski problem {\it w.r.t.} $2$-torsional rigidity by method of curvature flows and obtain the existence of smooth even solutions. Up to now, as far as we know, the study of the Minkowski problem for the $q$-torsional rigidity is still  blank.

In the present paper, we propose and investigate the Orlicz Minkowski problem for the $q$-torsional rigidity corresponding to the $q$-Laplace equation inspired by the foregoing works, and then confirm the existence of smooth non-even solutions to the Orlicz Minkowski problem for the $q$-torsional rigidity with $q>1$ by the method of a Gauss curvature flow.
\end{abstract}

\maketitle

\vskip 20pt
\section{Introduction and main results}

The classical Minkowski problem argues the existence, uniqueness and regularity of a convex body whose surface area measure is equal to a pre-given Borel measure on the sphere $S^{n-1}$. If the given measure has a positive continuous density, the Minkowski problem can be seen as the problem of prescribing the Gauss curvature in differential geometry. The Minkowski problem and its solution can be traced back to the works of Minkowski \cite{MH}, other influential works, such as Lewy \cite{LH}, Nirenberg \cite{NL}, Pogorelov \cite{PA} and Cheng-Yau \cite{CS}, etc..

In the past $30$ years, the Minkowski problem played an important role in the study of convex geometry, and the research of Minkowski problem has promoted the development of fully nonlinear partial differential equations. The Minkowski problem has produced some variations of it, among which the $L_p$($p\in \mathbb{R}$) Minkowski problem is particularly important because the $L_p$($p\in \mathbb{R}$) Minkowski problem contains some special versions. When $p=1$, it is the classical Minkowski problem; when $p=0$, it is the log-Minkowski problem \cite{BO}; when $p=-n$, it is the centro-affine Minkowski problem \cite{ZG}. The $L_p$ Minkowski problem with $p>1$ was first proposed and studied by Lutwak \cite{LE0}, whose solution plays a key role in establishing the $L_p$ affine Sobolev inequality \cite{HC2, LE01}. Until 2010, Haberl, Lutwak, Yang and Zhang \cite{HC1} extended the ($L_p$) Minkowski problem to Orlicz space which breaks through the limitation of $L_p$ space that can only handle polynomial growth functions by introducing the Young function $\varphi(s)$ as a ``growth metric", providing a powerful and flexible toolbox for analyzing functions with arbitrary (convex) growth behavior, heterogeneity, or anisotropy. It unifies and promotes the $L_p$ space theory, greatly expanding the application scope of function space methods, enabling them to more accurately describe and solve complex nonlinear problems in numerous natural science and engineering fields. And the Orlicz Minkowski problem is a profound paradigm extension in convex geometry analysis. It unifies the classical Minkowski problem and $L_p$ Minkowski problem into one framework by introducing the Young functions $\varphi$, greatly expanding the range and complexity of geometric anisotropy that can be described. It not only solves a wider range of geometric inverse problems, but also promotes convex geometry, PDE and the technological development in fields such as optimal transmission, has also given rise to new research directions (such as Orlicz Brunn Minkowski theory), providing powerful tools for understanding highly heterogeneous and anisotropic geometric structures.

We know that the different geometric measures are corresponding to the different Minkowski type problems. Some geometric measures with physical backgrounds have been introduced into the Brunn-Minkowski theory, naturally, the related Minkowski type problems have also been gradually studied. For instance, Xiao \cite{XJ} prescribed capacitary curvature measures on planar convex domains: If a given finite nonnegative Borel measure $\mu \in S^1$ has centroid at the origin and its supp($\mu$) does not comprise any pair of antipodal points, then there is a unique (up to translation) convex, nonempty, open set $\Omega\subset \mathbb{R}^2$ such that $d\mu_q(\Omega,\cdot)=d\mu(\cdot)$, where $\mu_q(\Omega,\cdot)$ is $q$-capacitary curvature measure of $\Omega$ with $q\in(1,2]$. In addition, the Minkowski problem for torsional rigidity ($2$-torsional rigidity) was firstly studied by Colesanti and Fimiani \cite{CA} using variational method. Moreover, the authors in \cite{ZX} posed the Orlicz chord Minkowski problem and obtained its smooth solutions by a Gauss curvature flow. Inspired by above works, we will continuous to study the Orlicz Minkowski type problem with a physical background in the present paper, namely, the Orlicz Minkowski problem for $q$-torsional rigidity.

For convenience, we recall and state firstly  the concept of $q$-torsional rigidity and its related contents as below. Let $\mathcal{K}^n$ be the collection of convex bodies in Euclidean space $\mathbb{R}^n$. The set of convex bodies containing the origin in their interiors in $\mathbb{R}^n$, we write $\mathcal{K}^n_o$. Moreover, we let $C^2_{+}$ be the class of convex bodies of $C^2$ if its boundary has the positive Gauss curvature.

We  now do the needful. The so-called torsional rigidity $T(\Omega)$(or $2$-torsional rigidity $T_2(\Omega)$) of convex body $\Omega$ in $\mathbb{R}^n$ is described by (see \cite{CA0})
\begin{align*}
\frac{1}{T(\Omega)}=\inf\bigg\{\frac{\int_{\Omega}|\nabla u|^2 dy}{[\int_{\Omega}|u|dy]^2}:u\in W_0^{1,2}(\Omega),\int_{\Omega}|u|dy>0\bigg\}.
\end{align*}
It has been provided that, there exists a unique function $u$ such that
\begin{align*}
T_2(\Omega)=\int_{\Omega}|\nabla u|^2dy,
\end{align*}
where $u$ satisfies the boundary-value problem

\begin{align*}
\left\{
\begin{array}{lc}
\Delta u(y)=-2\ \ \  \text{in} \ \ \ \Omega,\\
u(y)=0,\ \ \ \ \ \ \ \ \text{on} \ \ \ \ \partial\Omega.\\
\end{array}
\right.
\end{align*}
Here, $\Delta u$ is the Laplace operator.

Furthermore, we  introduce the $q$-torsional rigidity (denoted by $T_q(\Omega)$)\cite{CA1} with $q>1$.  Let $\Omega\in\mathcal{K}^n$, the $q$-torsional rigidity $T_q(\Omega)$  is defined by
\begin{align}
\label{eq101}\frac{1}{T_q(\Omega)}=\inf\bigg\{\frac{\int_{\Omega}|\nabla u|^q dy}{[\int_{\Omega}|u|dy]^q}:u\in W_0^{1,q}(\Omega),\int_{\Omega}|u|dy>0\bigg\}.
\end{align}
The functional defined in (\ref{eq101}) admits a minimizer $u\in W_0^{1,q}(\Omega)$, and $cu$( for some constant $c$) is unique positive solution of the following boundary value problem (see \cite{BM} or \cite{HH})
\begin{align}\label{eq102}
\left\{
\begin{array}{lc}
\Delta_qu=-1\ \ \  \text{in} \ \ \ \Omega,\\
u=0,\ \ \ \ \ \ \ \ \text{on} \ \ \ \ \partial\Omega,\\
\end{array}
\right.
\end{align}
where $$\Delta_qu={\rm div}(|\nabla u|^{q-2}\nabla u)$$
is the $q$-Laplace operator. Obviously, when $q=2$, it is the $2$-torsional rigidity.

Applying (\ref{eq102}) with the Gauss-Green formula, we have
\begin{align}\label{eq103}
\int_{\Omega}|\nabla u|^qdy=\int_{\Omega}udy,
\end{align}
from (\ref{eq101}) and (\ref{eq103}), it follows
\begin{align}\label{eq104}
T_q(\Omega)=\frac{(\int_{\Omega}udy)^q}{\int_{\Omega}|\nabla u|^qdy}=\bigg(\int_{\Omega}udy\bigg)^{q-1}=\bigg(\int_{\Omega}|\nabla u|^qdy\bigg)^{q-1}.
\end{align}
With the aid of Poho$\check{\textrm{z}}$aev-type identities of \cite{PU}, the $q$-torsional rigidity formula (\ref{eq104}) is given by
\begin{align}\label{eq105}
T_q(\Omega)^{\frac{1}{q-1}}=&\frac{q-1}{q+n(q-1)}\int_{S^{n-1}}h(\Omega,x)d\mu^{tor}_{q}(\Omega,x)\\
\nonumber=&\frac{q-1}{q+n(q-1)}\int_{S^{n-1}}h(\Omega,x)|\nabla u|^qdS(\Omega,x).
\end{align}
Denoting $\widetilde{T}_q(\Omega)=T_q(\Omega)^{\frac{1}{q-1}}$, the $q$-torsional measure $\mu^{tor}_{q}(\Omega,\eta)$ is defined by
\begin{align}\label{eq106}
\mu^{tor}_q(\Omega,\eta)=\int_{g^{-1}(\eta)}|\nabla u(y)|^qd\mathcal{H}^{n-1}(y)=\int_\eta|\nabla u(g^{-1}(x))|^qdS(\Omega,x),
\end{align}
for any Borel set $\eta\subseteq S^{n-1}$. Here, $g:\partial\Omega\rightarrow S^{n-1}$ is a Gauss map, and $\mathcal{H}^{n-1}$ is the $(n-1)$-dimensional Hausdorff measure.

Recently, Hu and Zhang \cite{HJ2} firstly established the functional form of the Orlicz Brunn-Minkowski theory for the $q$-torsional rigidity with $q>1$ by variational formula in the smooth category. Combining the Orlicz Minkowski sum with formula (\ref{eq105}), they introduced the Orlicz mixed $q$-torsional rigidity of functional form as follows.

\begin{definition}\label{def11}\cite[Definition 3.4]{HJ2}~~~~Suppose $\varphi\in \overline{\Phi}$, $q>1$, and $f, g,a\cdot f+_{\varphi} b\cdot g\in \mathcal{E}$ for $a,b\in I$ (not both zero). Define the Orlicz mixed $q$-torsional rigidity $\widetilde{T}_{\varphi,q}([f],g)$ by
\begin{align*}
\widetilde{T}_{\varphi,q}([f],g)=\frac{\gamma\varphi^\prime_l(1)}{\alpha}\frac{d\widetilde{T}_q(f+_{\varphi} t\cdot g)}{dt}\bigg|_{t=0^+}=\frac{\gamma}{\alpha}\int_{S^{n-1}}\varphi\bigg(\frac{g(x)}{f(x)}\bigg)f(x)d\mu^{tor}_q([f],x).
\end{align*}
Here,
\begin{align*}\overline{\Phi}=\{\varphi \in C^{2,\alpha}(\mathbb{R}):\varphi \in \Phi:\Phi~~\text {be the class of convex and strictly increasing function}\},
\end{align*}
\begin{align*}
\mathcal{E}=\{h\in C^{2,\alpha}_{+}(S^{n-1}):(h_{ij}+h\delta_{ij})~~ \text{is positive definite}\},
\end{align*}
$I\subset[0,+\infty)$ is a bounded interval and $\frac{\gamma}{\alpha}=\frac{q-1}{n(q-1)+q}$.
\end{definition}
Suppose $K,L\in\mathcal{K}^n_o, q>1$, and  are of class $C^{2,\alpha}_+$, analogous to Definition \ref{def11}, there also has following definition.

\begin{definition}\label{def12}\cite[Definition 3.8]{HJ2}~~~~Suppose $\varphi\in \overline{\Phi}$, $q>1$, and $K, L,a\cdot K+_{\varphi} b\cdot L\in \mathcal{K}^n_o$ that are of class $C^{2,\alpha}_+$ for $a,b\in I$ (not both zero). Then the Orlicz mixed $q$-torsional rigidity $\widetilde{T}_{\varphi,q}(K,L)$ defined by
\begin{align*}
\widetilde{T}_{\varphi,q}(K,L)=\frac{\gamma\varphi^\prime_l(1)}{\alpha}\frac{d\widetilde{T}_q(K+_{\varphi} t\cdot L)}{dt}\bigg|_{t=0^+}=\frac{\gamma}{\alpha}\int_{S^{n-1}}\varphi\bigg(\frac{h_L(x)}{h_K(x)}\bigg)h_K(x)d\mu^{tor}_q(K,x).
\end{align*}
\end{definition}
We can see if $\varphi(s)=s^p$ $(1\leq s<\infty)$, Definition \ref{def12} reduces to the $L_p$ variational formula for $q$-torsional rigidity in \cite{HY0}.

Motivated by the forgoing  variational formula for $q$-torsional rigidity with respect to Orlicz sum, we not only get the Orlicz mixed $q$-torsional rigidity of $K, L$, but also obtain Orlicz $q$-torsional measure $\mu^{tor}_{\varphi,q}$. Thus, we can propose the normalised Orlicz Minkowski problem for $q$-torsional rigidity.

{\bf Orlicz Minkowski problem for $q$-torsional rigidity}: Let $q>1$, $\mu$ be a given nonzero finite Borel measure on $S^{n-1}$, $\varphi\in\overline{\Phi}$ and $\varphi:(0,\infty)\rightarrow (0,\infty)$, under what the necessary and sufficient conditions on $\mu$, does there exist an unique convex body $\Omega$ whose support function is $h$ and a positive constant $\tau$ so that $\mu=\tau\mu^{tor}_{\varphi,q}$, i.e.
\begin{align}\label{eq107}
d\mu=\tau\varphi(h)d\mu^{tor}_{q}(\Omega,\cdot).
\end{align}

Combining (\ref{eq106}), if the given measure $\mu$ is absolutely continuous with respect to the Lebesgue measure and $\mu$ has a smooth density function $f:S^{n-1}\rightarrow (0,\infty)$, then solving problem (\ref{eq107}) can be equivalently viewed as solving the following Monge-Amp\`{e}re equation on $S^{n-1}$:
\begin{align}\label{eq108}
\tau\varphi(h)|\nabla u|^q\det(h_{ij}+h\delta_{ij})=f.
\end{align}

In (\ref{eq107}), when $q=2$, the Orlicz Minkowksi problem w.r.t. the $2$-torsional rigidity was first developed and proven by Li and Zhu \cite{LN}. Moreover, Hu, Liu and Ma \cite{HJ} with the help of Equation (\ref{eq108}) obtained the existence of smooth even solutions by Gauss curvature flow for this problem.
In the case of $q=2$, let $\varphi(s)=s$, the Minkowski problem for $2$-torsional rigidity was firstly proposed by Colesanti and Fimiani \cite{CA} and the smooth even solution was obtained in \cite{HJ00} by a Gauss curvature flow. The same case of $q$, the Minkowski problem for $2$-torsional rigidity was extended $L_p$ version ($\varphi(s)=s^{p}$) by Chen and Dai \cite{CZ} who proved the existence of solutions for any fixed $p>1$ and $p\neq n+2$, Hu and Liu \cite{HJ01} for $0<p<1$.

In the present paper, we will continuous to investigate the Orlicz Minkowski problem for $q$-torsional rigidity by method of the curvature flows. The curvature flow is a very important method and an effective tool for studying geometric problems. For example, the mean curvature flow plays a crucial role in the study of geometric inequalities. Wang, Weng and Xia \cite{WG} constructed a new locally constrained curvature flow to obtain Alexandrov-Fenchel inequalities for convex hypersurfaces in the half-space with capillary boundary. Wei and Xiong \cite{WY} gave new proof of a class of Alexandrov-Fenchel
inequalities for anisotropic mixed volumes of smooth convex domains in Euclidean space by a volume-preserving anisotropic mean curvature type flow, etc.. In addition, the curvature flows also involving the particularly essential Gauss curvature flow which was first introduced and studied by Firey \cite{FI} to model the shape change of worn stones. It can mainly be used to study the existence of smooth solutions to the famous Minkowski (type) problems. Chen, Huang and Zhao \cite{CC} obtained smooth even solutions to the $L_p$ dual Minkowski problem by the Gauss curvature flow. Liu and Lu \cite{LY} solved dual Orlicz-Minkowski problem by the Gauss curvature flow and obtained smooth solutions. Various Gauss curvature flows have been extensively studied, see \cite{AB, BR, BR1, CH, CK0, HJ00, HJ, LY1} and the references therein.

In this article, we investigate the Orlicz Minkowski problem for $q$-torsional rigidity with $q>1$ and confirm the existence of smooth, non-even strictly convex solutions for (\ref{eq108}) by a Gauss curvature flow. Let $\partial\Omega_0$ be a smooth, closed and strictly convex hypersurface in $\mathbb{R}^n$ containing the origin in its interior, $f$ is a positive smooth function on $S^{n-1}$, $\varphi\in\overline{\Phi}$ and $\varphi:(0,\infty)\rightarrow (0,\infty)$. We construct and consider the long-time existence and convergence of the following Gauss curvature
flow which is a family of convex hypersurfaces $\partial\Omega_t$  parameterized by  smooth maps $X(\cdot ,t):
S^{n-1}\times (0, \infty)\rightarrow \mathbb{R}^n$
satisfying the initial value problem
\begin{align}\label{eq109}
\left\{
\begin{array}{lc}
\frac{\partial X(x,t)}{\partial t}=-\lambda(t)f(\nu)\frac{ (X\cdot \nu)}{|\nabla u(X,t)|^q\varphi(X\cdot \nu)}\mathcal{K}
(x,t)\nu+X(x,t),  \\
X(x,0)=X_0(x),\\
\end{array}
\right.
\end{align}
where $\mathcal{K}(x,t)$ is the Gauss curvature of hypersurface $\partial\Omega_t$,  $\nu=x$ is the
outer unit normal at $X(x,t)$, $X\cdot \nu $ represents standard inner product of $X$ and $\nu$, and $\lambda(t)$ is defined as
\begin{align*}\lambda(t)=\frac{\int_{S^{n-1}}|\nabla u(X,t)|^q\rho^nd\xi}{\int_{S^{n-1}}\frac{hf(x)}{\varphi(h)}dx},\end{align*}
where $\rho$ and $h$ are the radial function and the support function of convex hypersurface $\partial \Omega_t$, respectively.

For the convergence of discussing Gauss curvature flow (\ref{eq109}) in the following text, we introduce a functional for any $t\geq 0$,
\begin{align}\label{eq110}
\Gamma(\Omega_t)=
\int_{S^{n-1}}f(x)\phi(h(x,t))dx.
\end{align}
Here, let's assume that $\phi(s)=\int_{0}^{s}\frac{1}{\varphi(\zeta)}d\zeta$ exists for all $s>0$ and $\lim_{s\rightarrow \infty}\phi(s)=\infty$, where $h(\cdot,t)$ is the support function of $\Omega_t$.

Combining problem (\ref{eq108}) with flow (\ref{eq109}), we establish the following main conclusion in this article.

\begin{theorem}\label{thm13}
Let $q>1$, $\varphi\in\overline{\Phi}$ and $\varphi:(0,\infty)\rightarrow (0,\infty)$, $u(X,t)$ be a solution of (\ref{eq102}) in $\Omega_t$, $\partial\Omega_0$ be a smooth, closed and strictly convex hypersurface in $\mathbb{R}^n$ containing the origin in its interior and $f$ is positive smooth function on $S^{n-1}$. Then the flow (\ref{eq109}) has an unique smooth non-even solution $\partial\Omega_t=X(S^{n-1},t)$ for $t\in(0,\infty)$. Moreover, when $t\rightarrow\infty$, there is a subsequence of $\partial\Omega_t$ converges in $C^{\infty}$ to a smooth, closed and strictly convex hypersurface $\Omega_\infty$ whose support function satisfies (\ref{eq108}).
\end{theorem}

This paper is organized as follows. We collect some necessary background materials in Section \ref{sec2}. In Section \ref{sec3}, we give the scalar form  of flow (\ref{eq109}) by
the support function and discuss the properties of two important geometric functionals along the flow (\ref{eq109}).
In Section \ref{sec4}, we give the priori estimates for the solution to the flow (\ref{eq109}). We obtain the
convergence of the flow (\ref{eq109}) and complete the proof of Theorem \ref{thm13} in Section \ref{sec5}.

\section{\bf Preliminaries}\label{sec2}
In this subsection, we give a brief review of some relevant notions and terminologies about
convex bodies and recall some basic properties of convex hypersurfaces that one can  refer
to \cite{UR} and a book of Schneider \cite{SC} for details.
\subsection{Convex bodies} Let $\mathbb{R}^n$ be the $n$-dimensional Euclidean space
and $\partial\Omega$ be a smooth, closed and strictly convex hypersurface containing the origin in its interior. The support function of convex body $\Omega$ enclosed by $\partial\Omega$ is defined by
\begin{align*}h_\Omega(\xi)=h(\Omega,\xi)=\max\{\xi\cdot y:y\in\Omega\},\quad \forall\xi\in S^{n-1},\end{align*}
and the radial function of $\Omega$ with respect to $o$ (origin) $\in\mathbb{R}$ is defined by
\begin{align*}\rho_{\Omega}(v)=\rho(\Omega,v)=\max\{c>0:cv\in\Omega\},\quad  v\in S^{n-1}.\end{align*}
The volume $Vol(\Omega)$ of $\Omega$ is defined by
\begin{align}\label{eq201}Vol(\Omega)=\frac{1}{n}\int_{S^{n-1}}\rho(\Omega,v)^ndv=\frac{1}{n}\int_{S^{n-1}}h(\Omega,\xi)dS(\Omega,\xi).\end{align}

For a compact convex subset $\Omega\in \mathcal{K}^n$ and $\xi\in S^{n-1}$, the intersection of a
supporting hyperplane with $\Omega$, $H(\Omega,\xi)$ at $\xi$ is given by
\begin{align*}H(\Omega,\xi)=\{y\in \Omega:y\cdot \xi=h_\Omega(\xi)\}.\end{align*}
A boundary point of $\Omega$ which only has one supporting hyperplane is called a regular point, otherwise, it is a singular point. The set of singular points is denoted as $\sigma \Omega$, it is
well known that $\sigma \Omega$ has spherical Lebesgue measure 0.

For $y\in\partial \Omega\setminus \sigma \Omega$, its Gauss map $g_\Omega:y\in\partial \Omega\setminus \sigma \Omega\rightarrow S^{n-1}$ is represented by
\begin{align*}g_\Omega(y)=\{\xi\in S^{n-1}:y\cdot \xi=h_\Omega(\xi)\}.\end{align*}
Correspondingly, for a Borel set $\eta\subset S^{n-1}$, its inverse Gauss map is denoted by
$g_\Omega^{-1}$,
\begin{align*}g_\Omega^{-1}(\eta)=\{y\in\partial \Omega:g_\Omega(y)\in\eta\}.\end{align*}
Specially, for a convex hypersurface $\partial\Omega$ of class $C^2$, then the support function of $\Omega$ can be stated as
\begin{align*}
h(\Omega,x)=x\cdot g^{-1}(x)=g(X(x))\cdot X(x), \quad X(x)\in \partial\Omega.
\end{align*}
Moreover, the gradient of $h(\Omega, \cdot)$ satisfies
\begin{align}\label{eq202}
\nabla h(\Omega,x)=g^{-1}(x),
\end{align}
where $\nabla$ is the gradient on $\mathbb{R}^n$.

For the Borel set $\eta\subset S^{n-1}$, its surface area measure is defined as
\begin{align*}S_\Omega(\eta)=\mathcal{H}^n(g_\Omega^{-1}(\eta)),\end{align*}
where $\mathcal{H}^n$ is the $n$-dimensional Hausdorff measure.

\subsection{Gauss curvature on convex hypersurface }~~Suppose that $\Omega$ is parameterized by the inverse
Gauss map $X:S^{n-1}\rightarrow \Omega$, that is $X(x)=g_\Omega^{-1}(x)$. Then the support function $h$ of $\Omega$ can be computed by
\begin{align}\label{eq203}h(x)=x\cdot X(x) , \ \ x\in S^{n-1},\end{align}
where $x$ is the outer normal of $\Omega$ at $X(x)$. Let $\overline{\nabla}$ be the gradient on $S^{n-1}$ and $\{e_1, e_2, \cdots, e_{n-1}\}$ be an orthonormal frame on $S^{n-1}$, denote  $e_{ij}$ by the standard metric on the sphere $S^{n-1}$.
Differentiating (\ref{eq203}), there has
\begin{align*}\overline{\nabla}_ih=\overline{\nabla}_ix\cdot X(x)+ x\cdot \overline{\nabla}_iX(x),\end{align*}
since $\overline{\nabla}_iX(x)$ is tangent to $\Omega$ at $X(x)$, thus,
\begin{align*}\overline{\nabla}_ih=\overline{\nabla}_ix\cdot X(x).\end{align*}

By differentiating (\ref{eq203}) twice, the second fundamental form $A_{ij}$ of $\Omega$ can be computed in terms of the support function,
\begin{align}\label{eq204}A_{ij} = \overline{\nabla}_{ij}h + he_{ij},\end{align}
where $\overline{\nabla}_{ij}=\overline{\nabla}_i\overline{\nabla}_j$ denotes the second order covariant derivative with respect to $e_{ij}$. The induced metric matrix $g_{ij}$ of $\Omega$ can be derived by Weingarten's formula,
\begin{align}\label{eq205}e_{ij}=\overline{\nabla}_ix\cdot \overline{\nabla}_jx= A_{ik}A_{lj}g^{kl}.\end{align}
The principal radii of curvature are the eigenvalues of the matrix $b_{ij} = A^{ik}g_{jk}$.
When considering a smooth local orthonormal frame on $S^{n-1}$, by virtues of (\ref{eq204})
and (\ref{eq205}), there is
\begin{align}\label{eq206}b_{ij} = A_{ij} = \overline{\nabla}_{ij}h + h\delta_{ij}.\end{align}
Then the Gauss curvature $\mathcal{K}(x)$ of $X(x)\in\partial\Omega$ is given by
\begin{align}\label{eq207}\mathcal{K}(x) = (\det (\overline{\nabla}_{ij}h + h\delta_{ij} ))^{-1}.\end{align}

\section{\bf Geometric flow and its associated functionals}\label{sec3}
In this subsection, we will introduce the geometric flow and its associated functionals for solving the Orlicz Minkowski problem for $q$-torsional rigidity with $q>1$. For convenience, the Gauss curvature flow is restated
here. Let $\partial\Omega_0$ be a smooth, closed and strictly convex hypersurface in $\mathbb{R}^n$ containing the origin in its interior, $f$ be a positive smooth function on $S^{n-1}$. We consider the following Gauss curvature flow
\begin{align}\label{eq301}
\left\{
\begin{array}{lc}
\frac{\partial X(x,t)}{\partial t}=-\lambda(t)f(\nu)\frac{(X\cdot \nu)}{|\nabla u(X,t)|^q\varphi(X\cdot \nu)}\mathcal{K}
(x,t)\nu+X(x,t),  \\
X(x,0)=X_0(x),\\
\end{array}
\right.
\end{align}
where $\mathcal{K}(x,t)$ is the Gauss curvature of the hypersurface $\partial\Omega_t$ at $X(\cdot,t)$, $\nu=x$ is the unit outer
normal vector of $\partial\Omega_t$ at $X(\cdot,t)$, $X\cdot \nu $ represents standard inner product of $X$ and $\nu$, and $\lambda(t)$ is given by
\begin{align}\label{eq302}\lambda(t)=\frac{\int_{S^{n-1}}|\nabla u(X,t)|^q\rho^nd\xi}{\int_{S^{n-1}}\frac{hf}{\varphi}dx}.\end{align}

Taking the scalar product of both sides of the equation and of the initial condition in
(\ref{eq301}) by $\nu$, by means of the definition of support function (\ref{eq203}) and formula (\ref{eq202}) , we describe the flow (\ref{eq301}) with the support function as follows
\begin{align}\label{eq303}\left\{
\begin{array}{lc}
\frac{\partial h(x,t)}{\partial t}=-\lambda(t)f(x)\frac{h(x,t)}{|\nabla u(\nabla h,t)|^q\varphi(h)}\mathcal{K}
(x,t)+h(x,t),  \\
h(x,0)=h_0(x).\\
\end{array}
\right.\end{align}

Next, we investigate the characteristics of two essential geometric functionals with respect to Eq. (\ref{eq303}). Firstly, we show the $q$-torsional rigidity unchanged along the flow (\ref{eq301}). The conclusion can be stated as the following lemma.

\begin{lemma}\label{lem31} Let $q>1$, the $q$-torsional rigidity $T_q(\Omega_t)$ is unchanged along the flow (\ref{eq301}) for $t\in[0,T)$, i.e.
\begin{align*}T_q(\Omega_t)=T_q(\Omega_0).\end{align*}
\end{lemma}

\begin{proof}
Let $h(\cdot,t)$ and $\rho(\cdot,t)$ be the support function and radial function of $\Omega_t$, respectively. $u(X,t)$ is the solution of (\ref{eq102}) in $\Omega_t$. The proposition 2.5 in \cite{HY0} tells us that
\begin{align*}\frac{\partial}{\partial t}T_q(\Omega_t)=&\frac{\partial}{\partial t}\bigg[\bigg(\frac{q-1}{q+n(q-1)}\int_{S^{n-1}}h(x,t)|\nabla u|^q\mathcal{K}^{-1}dx\bigg)^{q-1}\bigg]\\
=&(q-1)T_q^{\frac{q-2}{q-1}}\int_{S^{n-1}}\frac{\partial h(x,t)}{\partial t}|\nabla u|^q\mathcal{K}^{-1}dx.
\end{align*}
Thus, from (\ref{eq302}), (\ref{eq303}) and $\rho^n\mathcal{K}d\xi=hdx$, we have

\begin{align*}
\frac{\partial}{\partial t}T_q(\Omega_t)=&(q-1)T_q^{\frac{q-2}{q-1}}\int_{S^{n-1}}\bigg(-\lambda(t)\frac{fh\mathcal{K}}{|\nabla u|^q\varphi}+h\bigg)|\nabla u|^q\mathcal{K}^{-1}dx\\
=&(q-1)T_q^{\frac{q-2}{q-1}}\bigg(-\frac{\int_{S^{n-1}}|\nabla u|^q\rho^nd\xi}{\int_{S^{n-1}}hf/\varphi dx}\int_{S^{n-1}}\frac{fh\mathcal{K}}{|\nabla u|^q\varphi}|\nabla u|^q\mathcal{K}^{-1}dx\\
&+\int_{S^{n-1}}h|\nabla u|^q\mathcal{K}^{-1}dx\bigg)\\
=&(q-1)T_q^{\frac{q-2}{q-1}}\bigg(-\int_{S^{n-1}}|\nabla u|^qh\mathcal{K}^{-1}dx+\int_{S^{n-1}}|\nabla u|^qh\mathcal{K}^{-1}dx\bigg)\\
=&0.
\end{align*}
This ends the proof of Lemma \ref{lem31}.
\end{proof}

The next lemma will show that the functional (\ref{eq110}) is non-increasing along the flow (\ref{eq301}).
\begin{lemma}\label{lem32}
The functional (\ref{eq110}) is non-increasing along the flow (\ref{eq301}). Namely, $\frac{\partial}{\partial t}\Gamma(\Omega_t)\leq0$, the equality holds if and only if $\Omega_t$ satisfies (\ref{eq108}).
\end{lemma}
\begin{proof}
By (\ref{eq110}), (\ref{eq302}), (\ref{eq303}), $\rho^n\mathcal{K}d\xi=hdx$ and the H\"{o}lder inequality, we obtain the following result,
\begin{align*}
&\frac{\partial}{\partial t}\Gamma(\Omega_t)\\
=&\int_{S^{n-1}}f(x) \phi'(h(x,t))\frac{\partial h}{\partial t}dx\\
=&\int_{S^{n-1}}\bigg(-\lambda(t)\frac{f(x)h}{|\nabla u|^q\varphi(h)}\mathcal{K}
+h\bigg)\frac{f(x)}{\varphi(h)}dx\\
=&-\lambda(t)\int_{S^{n-1}}\frac{f^{2}(x)h}{|\nabla u|^q\varphi^2(h)}\mathcal{K}dx+\int_{S^{n-1}}\frac{hf(x)}{\varphi(h)}dx\\
=&-\frac{\int_{S^{n-1}}|\nabla u|^q\frac{h}
{\mathcal{K}}dx}{\int_{S^{n-1}}\frac{hf}{\varphi(h)}dx}\int_{S^{n-1}}\frac{f^{2}(x)h}{|\nabla u|^q\varphi^2(h)}\mathcal{K}dx+\int_{S^{n-1}}\frac{hf(x)}{\varphi(h)}dx\\
=&\bigg(\int_{S^{n-1}}\frac{hf}{\varphi(h)}dx\bigg)^{-1}\bigg\{-\bigg[\bigg(\int_{S^{n-1}}\bigg(|\nabla u|^{\frac{q}{2}}(\frac{h}{\mathcal{K}})^{\frac{1}{2}}\bigg)^2dx\bigg)^{\frac{1}{2}}\bigg(\int_{S^{n-1}}\bigg(\frac{f(x)h^{\frac{1}{2}}\mathcal{K}^{\frac{1}{2}}}{|\nabla u|^{\frac{q}{2}}\varphi(h)}\bigg)^2dx\bigg)^{\frac{1}{2}}\bigg]^2\\
&+\bigg(\int_{S^{n-1}}\frac{hf(x)}{\varphi(h)}dx\bigg)^2\bigg\}\\
\leq&\bigg(\int_{S^{n-1}}\frac{hf(x)}{\varphi(h)}dx\bigg)^{-1}\bigg[\!-\bigg(\!\int_{S^{n-1}}|\nabla u|^\frac{q}{2}(\frac{h}{\mathcal{K}})^{\frac{1}{2}}\times \frac{f(x)h^{\frac{1}{2}}\mathcal{K}^\frac{1}{2}}{|\nabla u|^{\frac{q}{2}}\varphi(h)}dx\bigg)^2+\bigg(\!\int_{S^{n-1}}\frac{hf(x)}{\varphi(h)}dx\bigg)^2\bigg]\\
=&0.
\end{align*}

By the equality condition of H\"{o}lder inequality, we know that the above equality holds if and only if $f=\tau\varphi(h)\mathcal{K}^{-1}|\nabla u|^q$, i.e.,
$$\tau\varphi(h)|\nabla u|^q\det(\nabla_{ij}h+h\delta_{ij})=f.$$
Namely, $\Omega_t$ satisfies (\ref{eq108}) with $\frac{1}{\tau}=\lambda(t)$.
\end{proof}

\section{\bf Priori estimates}\label{sec4}

In this subsection, we establish the $C^0, C^1$ and $C^2$ estimates for the solution to Eq. (\ref{eq303}). In the following of this paper, we always assume that $\partial\Omega_0$ is a smooth, closed and strictly convex hypersurface in $\mathbb{R}^n$ containing the origin in its interior, $h:S^{n-1}\times [0,T)\rightarrow \mathbb{R}$ is a smooth non-even solution to Eq. (\ref{eq303}) with the initial $h(\cdot,0)$ the support function of $\partial\Omega_0$. Here, $T$ is the maximal time for the existence of smooth non-even solution to Eq. (\ref{eq303}).

\subsection{$C^0, C^1$  estimates}

In order to complete the $C^0$ estimate, we firstly need to introduce the following lemma.

\begin{lemma}\label{lem42} Let $u\in W_{loc}^{1,q}(\Omega)$ be a local weak solution of
\begin{align*}
{\rm div}(|\nabla u|^{q-2}\nabla u)=\psi, ~~ q>1; ~~ \psi\in L_m^{loc}(\Omega),
\end{align*}
$m>q^\prime n$ $(\frac{1}{q^\prime}+\frac{1}{q}=1)$. Then $u\in C^{1+\alpha}_{loc}(\Omega)$. (see \cite[Corollary in pp. 830]{DE})
\end{lemma}

\begin{lemma}\label{lem43}
Let $q>1$, $\Omega_t$ be a smooth, non-even and strictly convex solution to the flow (\ref{eq301}) in $\mathbb{R}^n$, $u(X,t)$ be a solution of (\ref{eq102}) in $\Omega_t$, $f$ and $\varphi$ be as in Theorem \ref{thm13}. Then there is a positive constant $C$ independent of
$t$ such that
\begin{align}\label{eq401}
\frac{1}{C}\leq h(x,t)\leq C, \ \ \forall(x,t)\in S^{n-1}\times[0,T),
\end{align}
\begin{align}\label{eq402}
\frac{1}{C}\leq \rho(\xi,t)\leq C, \ \ \forall(\xi,t)\in S^{n-1}\times[0,T).
\end{align}
Here, $h(x,t)$ and $\rho(\xi,t)$ are the support function and radial function of $\Omega_t$, respectively.
\end{lemma}
\begin{proof}
Due to  $\rho(\xi,t)\xi=\overline{\nabla} h(x,t)+h(x,t)x$. Clearly, one sees
\begin{align*}
\min_{S^{n-1}} h(x,t)\leq \rho (\xi,t)\leq \max_{S^{n-1}} h(x,t).
\end{align*}
This implies that the estimate (\ref{eq401}) is tantamount to the estimate (\ref{eq402}). Thus, we only need to estimate (\ref{eq401}) or (\ref{eq402}).

Firstly, we derive at the uniform upper bound of $\rho(\xi,t)$, and for this purpose, we need to prove the uniformly lower bound of $|\nabla u(y,t)|$. Here, we use the method of proof by contradiction. We assume there is a sequence $(y,t_k)\in \Omega\times [0,T)$ such that $|\nabla u(y,t_k)|=0$, but ${\rm div}(|\nabla u|^{q-2}\nabla u)=-1$ implies that $|\nabla u|\neq 0$, thus, the assumption $|\nabla u(y,t_k)|=0$ is not valid. Next, we suppose there exists a sequence $(y,t_k)\in \Omega\times [0,T)$ such that $|\nabla u(y,t_k)|\rightarrow 0$, denote $M=|\nabla u|^{q-2}\nabla u$, then ${\rm div}(M)=-1$ from (\ref{eq102}). When $q\geq2$, $|\nabla u|^{q-2}\rightarrow 0$, then $M$ tends to a zero vector. When $1<q<2$, $|\nabla u|^{q-2}\rightarrow +\infty$, but $|M|=|\nabla u|^{q-1}\rightarrow 0$ because of $0<q-1<1$, thus, $M\rightarrow \vec{0}$ with $1<q<2$, too. Therefore, whether $q\geq2$ or $1<q<2$, $M$ is always tends to a zero vector on $(y,t_k)$. Since ${\rm div}(M)=-1$, i.e. $\sum_{i=1}^{n}\frac{\partial M_i}{\partial y_i}=-1$, where $M_i=|\nabla u|^{q-2}\frac{\partial u}{\partial y_i}$. Because $\Omega$ is a convex body, from Lemma \ref{lem42}, we know that $\frac{\partial M_i}{\partial y_i}$ is bounded independent of $t$ in $\Omega$. However, when $M\rightarrow \vec{0}$, ${\rm div}(M)=\sum_{i=1}^{n}\frac{\partial M_i}{\partial y_i}\rightarrow {\rm div}(\vec{0})=0$ by continuity (Lemma \ref{lem42}), this a contradiction with ${\rm div}(M)=-1$. Thus, $|\nabla u(y,t)|$ has a uniformly positive lower bound $c$ independent of $t$.

From Lemma \ref{lem31} and $\rho^nd\xi=h\mathcal{K}^{-1}dx$, we know that
\begin{align*}
T_q(\Omega_0)=T_q(\Omega_t)=&\bigg(\frac{q-1}{q+n(q-1)}\int_{S^{n-1}}h(x,t)|\nabla u|^q\mathcal{K}^{-1}dx\bigg)^{q-1}\\
=&\bigg(\frac{q-1}{q+n(q-1)}\int_{S^{n-1}}\rho^n(\xi,t)|\nabla u|^qd\xi\bigg)^{q-1}.
\end{align*}
Then
\begin{align*}
T_q(\Omega_0)^{\frac{1}{q-1}}=&\frac{q-1}{q+n(q-1)}\int_{S^{n-1}}\rho^n(\xi,t)|\nabla u|^qd\xi\\
\geq&c^q\frac{q-1}{q+n(q-1)}\int_{S^{n-1}}\rho^n(\xi,t)d\xi.
\end{align*}
Here, we define $\rho_{\max_t}(\xi)$ (or $\rho_{\sup_t}(\xi)$) by the maximum value or supremum of $\rho(\xi,t)$ w.r.t. $t$. The foregoing inequality shows that $\rho_{\max_t}(\xi)<+\infty$ (or $\rho_{\sup_t}(\xi)<+\infty$). In fact, if $\rho_{\max_t}(\xi)=+\infty$, then $\int_{S^{n-1}}\rho^n_{\max_t}(\xi)d\xi=+\infty$ from the continuity of $\rho$ on $S^{n-1}$, this is a contradictory with
\begin{align*}
\int_{S^{n-1}}\rho^n_{\max_t}(\xi)d\xi\leq\frac{q+n(q-1)}{c^q(q-1)}T_q(\Omega_0)^{\frac{1}{q-1}}=\widetilde{C}~(\text{a positive constant}).
\end{align*}
The same conclusion also applies to $\rho_{\sup_t}(\xi)$. This implies that there is a positive constant $C$ such that $\rho_{\max_t}(\xi)\leq C<+\infty$ (or $\rho_{\sup_t}(\xi)\leq C<+\infty$), where $C$ is independent of $t$.

Next, we provide a proof to uniformly lower bound of $\rho(\xi,t)$, and for this purpose, we need to prove the uniformly upper bound of $|\nabla u(X(x,t),t)|$. We have obtained the uniform  upper bounds of the convex bodies generated by $\partial \Omega_t=X(S^{n-1},t)$. Thus, there ia a ball $B_R$ tangent to $\partial\Omega_t$ at $X(x)$ with $\Omega_t\subset B_R$ for $t\in[0,T)$ and $X(x)\in\partial\Omega_t$. Let $u_R$ be a solution of (\ref{eq102}) in $B_R$. By the comparison principle, $u(\cdot,t)\leq u_R(\cdot)$ in $\Omega_t$, and $|\nabla u(X(x,t),t)|\leq |\nabla u_R(X(x))|$. In addition, we know that $|\nabla u_R(X(x))|\leq \widehat{C}$ by Lemma \ref{lem42}. Consequently, there exists a positive constant $C_1$ independent of $t$ such that
\begin{align}\label{eq403}
|\nabla u (X(x,t),t)|\leq C_1, \quad \forall (x,t)\in S^{n-1}\times [0,T),
\end{align}
where $C_1$ is independent of $t$.

Similar to the proof of upper bound for $\rho(\xi,t)$, we define $\rho_{\min_t}(\xi)$ (or $\rho_{\inf_t}(\xi)$) by the minimum value or infimum of $\rho(\xi,t)$ w.r.t. $t$, then for any $x\in S^{n-1}$, there is
\begin{align*}
T_q(\Omega_0)^{\frac{1}{q-1}}=&\frac{q-1}{q+n(q-1)}\int_{S^{n-1}}\rho^n(\xi,t)|\nabla u|^qd\xi\\
\leq &C_1^q\frac{q-1}{q+n(q-1)}\int_{S^{n-1}}\rho^n(\xi,t)d\xi.
\end{align*}
The foregoing inequality shows that $\rho_{\min_t}(\xi)>0$ (or $\rho_{\inf_t}(\xi)>0$). In fact, if $\rho_{\min_t}(\xi)=0$, then $\int_{S^{n-1}}\rho^n_{\min_t}(\xi)d\xi=0$ from the continuity of $\rho$ on $S^{n-1}$, this is a contradictory with
\begin{align*}
\int_{S^{n-1}}\rho^n_{\min_t}(\xi)d\xi\geq\frac{q+n(q-1)}{C_1^q(q-1)}T_q(\Omega_0)^{\frac{1}{q-1}}=\tilde{c}~(\text{a positive constant}).
\end{align*}
The same conclusion also applies to $\rho_{\inf_t}(\xi)$. This implies that there is a positive constant $\delta$ such that $\rho_{\min_t}(\xi)\geq\delta>0$ (or $\rho_{\inf_t}(\xi)\geq\delta>0$), where $\delta$ is independent of $t$.
\end{proof}

\begin{lemma}\label{lem44}Let $q>1$, $\Omega_t$ be a smooth, non-even and strictly convex solution to the flow (\ref{eq301}) in $\mathbb{R}^n$, $u(X,t)$ be a solution of (\ref{eq102}) in $\Omega_t$, $f$ and $\varphi$ satisfy conditions of Theorem \ref{thm13}. Then there is a positive constant $C$ independent of $t$ such that
\begin{align}\label{eq404}|\overline{\nabla} h(x,t)|\leq C,\quad\forall(x,t)\in S^{n-1}\times [0,T),\end{align}
and
\begin{align}\label{eq405}|\overline{\nabla} \rho(\xi,t)|\leq C,\quad \forall(\xi,t)\in S^{n-1}\times [0,T).\end{align}
\end{lemma}

\begin{proof}
The desired results immediately follows from Lemma \ref{lem43} and the following identities (see e.g. \cite{LR})
\begin{align*}
h=\frac{\rho^2}{\sqrt{\rho^2+|\overline{\nabla}\rho|^2}},\qquad\rho^2=h^2+|\overline{\nabla} h|^2.\end{align*}
\end{proof}

\begin{lemma}\label{lem45} Under the same conditions as the Lemma \ref{lem43}, there always exists a positive constant $C$ independent of $t$, such that
\begin{align*}\frac{1}{C}\leq\lambda(t)\leq C,\quad t\in [0,T).\end{align*}
\end{lemma}

\begin{proof} From the definition of $\lambda(t)$, $\rho^nd\xi=h\mathcal{K}^{-1}dx$, (\ref{eq105}) and Lemma \ref{lem31}, we have
\begin{align*}
\lambda(t)=\frac{\int_{S^{n-1}}|\nabla u(X,t)|^q\rho^nd\xi}{\int_{S^{n-1}}\frac{hf}{\varphi}dx}=\frac{\frac{q+n(q-1)}{q-1}T_q(\Omega_t)^{\frac{1}{q-1}}}{\int_{S^{n-1}}\frac{hf}{\varphi(h)}dx}=\frac{\frac{q+n(q-1)}{q-1}T_q(\Omega_0)^{\frac{1}{q-1}}}{\int_{S^{n-1}}\frac{hf}{\varphi(h)}dx},
\end{align*}
the conclusion of this result is directly obtained from Lemma \ref{lem43}.
\end{proof}

\begin{remark}\label{rem46} From the proof of Lemma \ref{lem43}, we know that $|\nabla u(X(x,t),t)|$ has uniformly positive upper and lower bounds. In the same time, by virtue of Schauder’s theory (see example Chapter 6 in \cite{GI}), there is a positive constant $\widehat{C}$ independent of $t$ satisfying that
\begin{align}\label{eq406}
|\nabla^k u (X(x,t),t)|\leq \overline{C}, \quad \forall (x,t)\in S^{n-1}\times [0,T),
\end{align}
for all integer $k\geq 2$.
\end{remark}

\subsection{$C^2$ estimate}

In this subsection, we establish the upper and lower bounds of principal curvature. This will shows
that Eq. (\ref{eq303}) is uniformly parabolic. The technique used in this proof was first introduced by Tso \cite{TK} to derive the upper bound of the Gauss curvature. We begin with completing the following results which will be need in $C^2$ estimate.

\begin{lemma}\label{lem47} Let $\Omega_t$ be a convex body of $C_+^2$ in $\mathbb{R}^n$, and $u(X(x,t),t)$ be the solution of (\ref{eq102}) with $q>1$ in $\Omega_t$, then
\begin{flalign*}
\begin{split}
(i)&(\nabla^2u(X(x,t),t)e_i)\cdot e_j=-\mathcal{K}|\nabla u(X(x,t),t)|c_{ij}(x,t);\\
(ii)&(\nabla^2u(X(x,t),t)e_i)\cdot x=-\mathcal{K}|\nabla u(X(x,t),t)|_jc_{ij}(x,t);\\
(iii)&(\nabla^2u(X(x,t),t)x)\cdot x=\frac{1}{q-1}\bigg(\mathcal{K}|\nabla u|{\rm Tr}(c_{ij}(h_{ij}+h\delta_{ij}))-|\nabla u|^{2-q}\bigg).
\end{split}&
\end{flalign*}
Here, $e_i$ and $x$ are orthonormal frame and unite outer normal on $S^{n-1}$, $\cdot$ is standard inner product and $c_{ij}$ is the cofactor matrix of $(h_{ij}+h\delta_{ij})$ with $\sum_{i,j}c_{ij}(h_{ij}+h\delta_{ij})=(n-1)\mathcal{K}^{-1}$.
\end{lemma}
\begin{proof} Similar conclusions have been presented in some references, for example \cite{HJ}. Here, we will briefly state the proofs combining with our problem.

(i)~~Assume that $h(x,t)$ is the support function of $\Omega_t$ for $(x,t)\in S^{n-1}\times [0,T)$ and let $\iota=\frac{\partial h}{\partial t}$. Then $X(x,t)=h_ie_i+hx, \frac{\partial X(x,t)}{\partial t}=\dot{X}(x,t)=\frac{\partial}{\partial t}(h_ie_i+hx)=\iota_ie_i+\iota x$. $X_i(X,t)=(h_{ij}+h\delta_{ij})e_j$, let $h_{ij}+h\delta_{ij}=\omega_{ij}$, then $X_{ij}(x,t)=\omega_{ijk}e_k-\omega_{ij}x$, where $\omega_{ijk}$ is the corariant derivatives of $\omega_{ij}$.

From $u(X,t)=0$ on $\partial\Omega_t$, we can not difficult to obtain
$$\nabla u\cdot X_i=0,$$
and
$$((\nabla^2u)X_j)X_i+\nabla u X_{ij}=0.$$
It follows that
\begin{align}\label{eq407}
\omega_{ik}\omega_{jl}(((\nabla^2u)e_l)\cdot e_k)+\omega_{ij}|\nabla u|=0.
\end{align}
Multiplying both sides of (\ref{eq407}) by $c_{ij}$, we have
\begin{align*}
c_{ij}\omega_{ik}\omega_{jl}(((\nabla^2 u)e_l)\cdot e_k)+\det(h_{ij}+h\delta_{ij})|\nabla u|=0.
\end{align*}
Namely,
\begin{align*}
\delta_{jk}\det(h_{ik}+h\delta_{ik})\omega_{jl}(((\nabla^2 u)e_l)\cdot e_k)+\det(h_{ij}+h\delta_{ij})|\nabla u|=0.
\end{align*}
It yields
\begin{align*}
\omega_{ij}(((\nabla^2 u)e_i)\cdot e_j)+|\nabla u|=0,
\end{align*}
then
\begin{align*}
c_{ij}\omega_{ij}(((\nabla^2 u)e_i)\cdot e_j)+c_{ij}|\nabla u|=0,
\end{align*}
i.e.,
\begin{align*}
\mathcal{K}^{-1}(((\nabla^2 u)e_i)\cdot e_j)+c_{ij}|\nabla u|=0,
\end{align*}
thus,
\begin{align*}
((\nabla^2 u)e_i)\cdot e_j=-c_{ij}\mathcal{K}|\nabla u|.
\end{align*}
This gives proof of (i).

(ii)~~Recall that
\begin{align*}
|\nabla u(X(x,t),t)|=-\nabla u(X(x,t),t)\cdot x,
\end{align*}
taking the covariant of both sides for above formula, we obtain
\begin{align}\label{eq408}
|\nabla u|_j=-\nabla u\cdot e_j-(\nabla^2 u)X_j\cdot x=-\omega_{ij}((\nabla^2 u)e_i\cdot x).
\end{align}
Multiplying both sides of (\ref{eq408}) by $c_{lj}$ and combining
\begin{align*}
c_{lj}\omega_{ij}=\delta_{li}\det(h_{ij}+h\delta_{ij}).
\end{align*}
We conclude that
\begin{align*}
c_{ij}|\nabla u|_j=-\det(h_{ij}+h\delta_{ij})(\nabla^2u)e_i\cdot x.
\end{align*}
Hence,
\begin{align*}
((\nabla^2u)e_i)\cdot x=-\mathcal{K}c_{ij}|\nabla u|_j.
\end{align*}
This proves (ii).

(iii)~~From (\ref{eq102}), we know that
\begin{align*}-1={\rm div}(|\nabla u|^{q-2}\nabla u)=|\nabla u|^{q-2}(\Delta u+\frac{q-2}{|\nabla u|^2}(\nabla^2u\nabla u)\cdot\nabla u),
\end{align*}
then
\begin{align*}
\frac{q-2}{|\nabla u|^2}(\nabla^2 u\nabla u )\cdot\nabla u=-\Delta u-|\nabla u|^{2-q},
\end{align*}
further,
\begin{align*}
&(q-2)((\nabla^2 u) x) \cdot x \\
&=-\Delta u-|\nabla u|^{2-q}\\
&=-{\rm Tr}(\nabla^2u)-|\nabla u|^{2-q}\\
&=-\sum_{i}((\nabla^2u) e_i) \cdot e_j-((\nabla^2u)x) \cdot x-|\nabla u|^{2-q}\\
&=\mathcal{K}|\nabla u|{\rm Tr}(c_{ij}(h_{ij}+h\delta_{ij}))-((\nabla^2u)x)\cdot x-|\nabla u|^{2-q},
\end{align*}
hence,
\begin{align*}
(q-1)((\nabla^2u)x) \cdot x=\mathcal{K}|\nabla u|{\rm Tr}(c_{ij}(h_{ij}+h\delta_{ij}))-|\nabla u|^{2-q},
\end{align*}
consequently,
\begin{align*}
((\nabla^2u) x) \cdot x=\frac{1}{q-1}\bigg(\mathcal{K}|\nabla u|{\rm Tr}(c_{ij}(h_{ij}+h\delta_{ij}))-|\nabla u|^{2-q}\bigg).
\end{align*}
This completes the proof of (iii).
\end{proof}

By Lemma \ref{lem43} and Lemma \ref{lem44}, if $h$ is a smooth non-even solution of Eq. (\ref{eq303}) on $S^{n-1}\times [0,T)$ and $f, \varphi$ satisfying conditions of Theorem \ref{thm13}, then along the flow for $[0,T), \overline{\nabla} h+hx$, and $h$ are smooth functions whose
ranges are within some bounded domain $\Omega_{[0,T)}$ and bounded interval $I_{[0,T)}$, respectively. Here, $\Omega_{[0,T)}$ and $I_{[0,T)}$ depend only
on the upper and lower bounds of $h$ on $[0,T)$.

\begin{lemma}\label{lem48} Let $q>1$, $\Omega_t$ be a smooth, non-even and strictly convex solution to the flow (\ref{eq301}) in $\mathbb{R}^n$, $u(X,t)$ be a solution of (\ref{eq102}) in $\Omega_t$, $f$ and $\varphi$ be as Theorem \ref{thm13} and $q>1$. Then there is a positive constant $C$ independents $t$ such that the principal curvatures $\kappa_i$ of $\Omega_t$, $i=1,\cdots, n-1$, are bounded from above and below, satisfying
\begin{align*}\frac{1}{C}\leq \kappa_i(x,t)\leq C, \quad\forall (x,t)\in S^{n-1}\times [0,T),\end{align*}
where $C$ depends on $\|f\|_{C^0(S^{n-1})}$, $\|f\|_{C^1(S^{n-1})}$, $\|f\|_{C^2(S^{n-1})}$, $\|\varphi\|_{C^1(I_{[0,T)})}$, $\|\varphi\|_{C^2(I_{[0,T)})}$, \\ $\|h\|_{C^0(S^{n-1}\times [0,T))}$, $\|h\|_{C^1(S^{n-1}\times [0,T))}$ and $\|\lambda\|_{C^0(S^{n-1}\times [0,T))}$.
\end{lemma}

\begin{proof} The proof is divided into two parts: in the first part, we derive an upper bound for the Gauss curvature $\mathcal{K}(x,t)$; in the second part, we give an estimate of bound above for the principal radii $b_{ij}=h_{ij}+h\delta_{ij}$.

Step 1: Prove $\mathcal{K}\leq C$.

Firstly, we construct the following auxiliary function,
\begin{align*}\Theta(x,t)=\frac{\lambda(t)(\varphi|\nabla u|^q)^{-1}\mathcal{K}f(x)h-h}{h-\varepsilon_0}
\equiv\frac{-h_t}{h-\varepsilon_0},\end{align*}
where
\begin{align*}\varepsilon_0=\frac{1}{2}\min_{S^{n-1}\times[0, T)}h(x,t)>0,\quad h_t =\frac{\partial h}{\partial t}.\end{align*}

For any fixed $t\in[0,T)$, we assume that $\Theta(x_0, t)=\max_{S^{n-1}}\Theta(x,t)$ is the spatial maximum of $\Theta$. Then at $(x_0,t)$, we have
\begin{align}\label{eq409}
0=\nabla_i\Theta=\frac{-h_{ti}}{h-\varepsilon_0}+\frac{h_th_i}{(h-\varepsilon_0)^2},
\end{align}
and from (\ref{eq409}), at $(x_0,t)$, we also get
\begin{align}\label{eq410}
\nonumber 0\geq\nabla_{ii}\Theta=&\frac{-h_{tii}}{h-\varepsilon_0}+\frac{h_{ti}h_{i}}{(h-\varepsilon_0)^2}
+\frac{h_{ti}h_{i}+h_th_{ii}}{(h-\varepsilon_0)^2}-\frac{h_th_i(2(h-\varepsilon_0)h_i)}{(h-\varepsilon_0)^4}\\
\nonumber=&\frac{-h_{tii}}{h-\varepsilon_0}+\frac{2h_{ti}h_{i}+h_th_{ii}}
{(h-\varepsilon_0)^2}
-\frac{2h_th_ih_i}{(h-\varepsilon_0)^3}\\
\nonumber=&\frac{-h_{tii}}{h-\varepsilon_0}+\frac{h_th_{ii}}
{(h-\varepsilon_0)^2}
+\frac{2h_{ti}h_i(h-\varepsilon_0)-2h_th_ih_i}{(h-\varepsilon_0)^3}\\
=&\frac{-h_{tii}}{h-\varepsilon_0}+\frac{h_th_{ii}}
{(h-\varepsilon_0)^2}.
\end{align}
From (\ref{eq410}), we obtain
$$-h_{tii}\leq\frac{-h_th_{ii}}{h-\varepsilon_0},$$
hence,
\begin{align}\label{eq411}\nonumber-h_{tii}-h_t\delta_{ii}\leq&\frac{-h_th_{ii}}{h-\varepsilon_0}-h_t\delta_{ii}=\frac{-h_t}
{h-\varepsilon_0}(h_{ii}+(h-\varepsilon_0)\delta_{ii})\\
=&\Theta(h_{ii}+h\delta_{ii}-\epsilon_0\delta_{ii})
=\Theta(b_{ii}-\varepsilon_0\delta_{ii}).\end{align}
At $(x_0,t)$, we also have
\begin{align}\label{eq412}\frac{\partial}{\partial t}\Theta=&\frac{-h_{tt}}{h-\epsilon_0}+\frac{h_t^2}{(h-\epsilon_0)^2}\\
\nonumber=&\frac{f}{h-\epsilon_0}\bigg[\frac{\partial(\lambda(t)(\varphi|\nabla u|^q)^{-1}h)}{\partial t}\mathcal{K}+\lambda(t)(\varphi|\nabla u|^q)^{-1}h\frac{\partial(\det(\overline{\nabla}^2h+hI))^{-1}}{\partial t}\bigg]+\Theta+\Theta^2,\end{align}
where
\begin{align*}
\frac{\partial}{\partial t}((\varphi|\nabla u|^q)^{-1}h)=&-\varphi^{-2}\varphi^{\prime}\frac{\partial h}{\partial t}|\nabla u|^{-q}h-q|\nabla u|^{-(q+1)}\frac{\partial}{\partial t}|\nabla u|\varphi^{-1} h+(\varphi|\nabla u|^q)^{-1}\frac{\partial h}{\partial t},
\end{align*}
and $\varphi^{\prime}$ denotes $\frac{\partial\varphi(s)}{\partial s}$.

According to Lemma 5.3 of \cite{HJ}, it shows that
\begin{align*}
\frac{\partial}{\partial t}|\nabla u|=&-(\nabla^2u)x\cdot \bigg(\frac{\partial h_i}{\partial t}\bigg)e_i-\bigg(\frac{\partial h}{\partial t}\bigg)(\nabla^2 u)x\cdot x-(|\nabla u|^{-1}\nabla u\nabla^2u\cdot x)\bigg(\frac{\partial h}{\partial t}\bigg)-|\nabla u|\bigg(\frac{\partial h}{\partial t}\bigg)\\
=&-(\nabla^2u)x\cdot \bigg(-\Theta_i(h-\epsilon_0)-\Theta h_i\bigg)e_i\\
&+\Theta(h-\epsilon_0)\bigg((\nabla^2 u)x\cdot x+|\nabla u|^{-1}\nabla u\nabla^2u\cdot x+|\nabla u|\bigg)\\
\leq & \Theta(x_0,t)\bigg((\nabla^2u)xh_ie_i+h((\nabla^2 u)x\cdot x+|\nabla u|^{-1}\nabla u\nabla^2u\cdot x+|\nabla u|)\bigg),
\end{align*}
then
\begin{align}\label{eq413}
\frac{\partial}{\partial t}&((\varphi|\nabla u|^q)^{-1}h)\\
\nonumber\leq&-\varphi^{-2}\varphi^{\prime}|\nabla u|^{-q}h(-\Theta(x_0,t)(h-\epsilon_0)\\
\nonumber&-q|\nabla u|^{-(q+1)}\varphi^{-1}h\Theta(x_0,t)\bigg((\nabla^2u)xh_ie_i+h((\nabla^2 u)x\cdot x+|\nabla u|^{-1}\nabla u\nabla^2u\cdot x+|\nabla u|)\bigg)\\
\nonumber&-(\varphi|\nabla u|^q)^{-1}\Theta(x_0,t)(h-\varepsilon_0).
\end{align}
Thus, combining (\ref{eq403}) with Lemma \ref{lem47}, and dropping some negative terms in (\ref{eq413}), we have
\begin{align*}
\frac{\partial}{\partial t}&((\varphi|\nabla u|^q)^{-1}h)\leq-\varphi^{-2}\varphi^{\prime}|\nabla u|^{-q}h\times(-\Theta(x_0,t))(h-\epsilon_0)\leq C_2\Theta(x_0,t).
\end{align*}
And from (\ref{eq302}) and Lemma \ref{lem31}, we know that
\begin{align*}\frac{\partial}{\partial t}(\lambda(t))=&\frac{\partial}{\partial t}\bigg(\frac{\int_{S^{n-1}}
|\nabla u|^q\rho^nd\xi}{\int_{S^{n-1}}hf/\varphi dx}\bigg)\\
=&\frac{\frac{\partial}{\partial t}
\bigg(\int_{S^{n-1}}|\nabla u|^q\rho^nd\xi\bigg)}{\int_{S^{n-1}}hf/\varphi dx}-\frac{\frac{\partial}{\partial t}\bigg(\int_{S^{n-1}}hf/\varphi dx\bigg)\bigg(\int_{S^{n-1}}|\nabla u|^q\rho^nd\xi\bigg)}{\bigg(\int_{S^{n-1}}hf/\varphi dx\bigg)^2}\\
=&-\frac{\frac{\partial}{\partial t}\bigg(\int_{S^{n-1}}hf/\varphi dx\bigg)\bigg(\int_{S^{n-1}}|\nabla u|^q\rho^nd\xi\bigg)}{\bigg(\int_{S^{n-1}}hf/\varphi dx\bigg)^2},\end{align*}
where
\begin{align*}-\frac{\partial}{\partial t}\bigg(\int_{S^{n-1}}hf/\varphi dx\bigg)=&-\int_{S^{n-1}}\frac{\frac{\partial h}{\partial t}f}{\varphi}+\frac{hf\varphi^\prime\frac{\partial h}{\partial t}}{\varphi^2}dx\\
=&\int_{S^{n-1}}\bigg(\frac{f}{\varphi}+\frac{hf\varphi^\prime}{\varphi^2}\bigg)\Theta(h-\varepsilon_0)dx\\
\leq&\Theta(x_0,t)\int_{S^{n-1}}\bigg(\frac{fh}{\varphi}+\frac{h^2f\varphi^\prime}{\varphi^2}\bigg)dx\\
\leq& C_3\Theta(x_0,t).
\end{align*}
Hence,
\begin{align*}\frac{\partial}{\partial t}(\lambda(t))\leq C_4\Theta(x_0,t).\end{align*}

We use (\ref{eq207}), (\ref{eq411}) and recall $b_{ij}=\overline{\nabla}_{ij}h+h\delta_{ij}$ may give
\begin{align*}\frac{\partial (\det(\overline{\nabla}^2h+hI))^{-1}}{\partial t}=&-(\det(\overline{\nabla}^2h+hI))^{-2}
\frac{\partial(\det(\overline{\nabla}^2h+hI))}{\partial b_{ij}}\frac{\partial(\overline{\nabla}^2h+hI)}{\partial t}\\
=&-(\det(\overline{\nabla}^2h+hI))^{-2}\frac{\partial(\det(\overline{\nabla}^2h+hI))}{\partial b_{ij}}
(h_{tij}+h_t\delta_{ij})\\
\leq&(\det(\overline{\nabla}^2h+hI))^{-2}\frac{\partial(\det(\overline{\nabla}^2h+hI))}{\partial b_{ij}}
\Theta(b_{ij}-\varepsilon_0\delta_{ij})\\
\leq&\mathcal{K}\Theta((n-1)-\varepsilon_0(n-1)\mathcal{K}^{\frac{1}{n-1}}).
\end{align*}
Therefore, we have following conclusion with (\ref{eq412}) at $(x_0,t)$,
\begin{align}\label{eq414}
\frac{\partial}{\partial t} \Theta\leq\frac{1}{h-\varepsilon_0}\bigg(C_5\Theta^2+f\lambda h(\varphi|\nabla u|^q)^{-1}\mathcal{K}\Theta((n-1)-\varepsilon_0(n-1)\mathcal{K}^{\frac{1}{n-1}})\bigg)+\Theta+\Theta^2.
\end{align}
According to construction of $\Theta$ and the  previous estimates, we easily obtain
\begin{align*}
\frac{1}{C_6}\mathcal{K}\leq \Theta\leq C_6\mathcal{K}.
\end{align*}
If $\Theta >>1$, then (\ref{eq414}) implies that
\begin{align*}
\frac{\partial}{\partial t}\Theta\leq& \frac{1}{h-\varepsilon_0}\bigg(C_5\Theta^2+f\lambda h(\varphi|\nabla u|^q)^{-1}C_6\Theta^2((n-1)-\varepsilon_0(n-1)(C_6\Theta) ^{\frac{1}{n-1}})\bigg)+\Theta+\Theta^2\\
\leq&\frac{1}{h-\varepsilon_0}\Theta^2\bigg(\!C_5\!+\![f\lambda h(\varphi|\nabla u|^q)^{-1}C_6(n-1)]\!-\![f\lambda h(\varphi|\nabla u|^q)^{-1}C_6^{\frac{n}{n-1}}(n-1)]\varepsilon_0\Theta ^{\frac{1}{n-1}}+2\!\bigg)\\
=&\frac{[f\lambda h(\varphi|\nabla u|^q)^{-1}C_6^{\frac{n}{n-1}}(n-1)]}{h-\varepsilon_0}\Theta^2\bigg(\frac{C_5+[f\lambda h(\varphi|\nabla u|^q)^{-1}C_6(n-1)]+2}{[f\lambda h(\varphi|\nabla u|^q)^{-1}C_6^{\frac{n}{n-1}}(n-1)]}-\varepsilon_0\Theta ^{\frac{1}{n-1}}\bigg)\\
\leq & C_7\Theta^2(C_8-\varepsilon_0\Theta^{\frac{1}{n-1}})<0,
\end{align*}
since $C_7$ and $C_8$ depend on $\|f\|_{C^0(S^{n-1})},$ $\|\varphi\|_{C^1(I_{[0,T)})}$,  $\|h\|_{C^0(S^{n-1}\times [0,T))}$, $\|h\|_{C^1(S^{n-1}\times [0,T))}$, $\|\lambda\|_{C^0(S^{n-1}\times [0,T))}$ and $|\nabla u|$. Consequently, we get
\begin{align*}
\Theta(x_0,t)\leq C,
\end{align*}
and for any $(x,t)$,
\begin{align*}
\mathcal{K}(x,t)=\frac{(h-\varepsilon_0)\Theta(x,t)+h}{f(x)h(\varphi|\nabla u|^q)^{-1}\lambda}\leq
\frac{(h-\varepsilon_0)\Theta(x_0,t)+h}{f(x)h(\varphi|\nabla u|^q)^{-1}\lambda}\leq C.
\end{align*}

Step 2: Prove $\kappa_i\geq\frac{1}{C}$.

We consider the auxiliary function as follows
\begin{align*}
\digamma (x,t)=\log\beta_{\max}(\{b_{ij}\})-A\log h+B|\overline{\nabla} h|^2,
\end{align*}
where $A, B$ are positive constants which will be chosen later, and $\beta_{\max}(\{b_{ij}\})$ denotes the maximal eigenvalue of $\{b_{ij}\}$; for convenience, we write $\{b^{ij}\}$ for $\{b_{ij}\}^{-1}$.

For every fixed $t\in[0,T)$, suppose $\max_{S^{n-1}}\digamma(x,t)$ is attained at point $x_0\in S^{n-1}$. By a rotation of coordinates, we may assume
\begin{align*}
\{b_{ij}(x_0,t)\} \text{ is diagonal,} \quad \text {and} \quad \beta_{\max}(\{b_{ij}\}(x_0,t))=b_{11}(x_0,t).
\end{align*}
Hence, in order to show $\kappa_i\geq\frac{1}{C}$, that is to prove $b_{11}\leq C.$ By means of the above assumption, we
transform $\digamma(x,t)$ into the following form,
\begin{align*}
\widetilde{\digamma}(x,t)=\log b_{11}-A\log h+B|\overline{\nabla} h|^2.
\end{align*}
Utilizing again the above assumption, for any fixed $t \in [0,T)$, $\widetilde{\digamma}(x,t)$ has a local
maximum at $(x_0,t)$, thus, we have at $(x_0,t)$,
\begin{align}\label{eq415}
0=\nabla_i\widetilde{\digamma}=&b^{11}\nabla_ib_{11}-A\frac{h_i}{h}+2B\sum h_kh_{ki}\\
\nonumber=&b^{11}(h_{i11}+h_1\delta_{i1})-A\frac{h_i}{h}+2Bh_ih_{ii},
\end{align}
and
\begin{align*}
0\geq&\nabla_{ii}\widetilde{\digamma}\\
=&\nabla_ib^{11}(h_{i11}+h_1\delta_{i1})+b^{11}
[\nabla_i(h_{i11}+h_1\delta_{i1})]-A\bigg(\frac{h_{ii}}{h}-\frac{h_i^2}{h^2}\bigg)
+2B(\sum h_kh_{kii}+h^2_{ii})\\
=&-(b_{11})^{-2}\nabla_ib_{11}(h_{i11}+h_1\delta_{i1})+b^{11}(\nabla_{ii}b_{11})-A\bigg(\frac{h_{ii}}{h}-\frac{h_i^2}{h^2}\bigg)
+2B(\sum h_kh_{kii}+h^2_{ii})\\
=&b^{11}\nabla_{ii}b_{11}-(b^{11})^2(\nabla_ib_{11})^2-A\bigg(\frac{h_{ii}}{h}-\frac{h_i^2}{h^2}\bigg)
+2B(\sum h_kh_{kii}+h^2_{ii}).	
\end{align*}
At $(x_0,t)$, we also have
\begin{align*}
\frac{\partial}{\partial t}\widetilde{\digamma}=&\frac{1}{b_{11}}\frac{\partial b_{11}}{\partial t}-A\frac{h_t}{h}+2B\sum h_kh_{kt}\\ =&b^{11}\frac{\partial}{\partial t}(h_{11}+h\delta_{11})-A\frac{h_t}{h}+2B\sum h_kh_{kt}\\
=&b^{11}(h_{11t}+h_t)-A\frac{h_t}{h}+2B\sum h_kh_{kt}.
\end{align*}

From Eq. (\ref{eq303}) and (\ref{eq207}), we know that
\begin{align}\label{eq416}
\nonumber \log(h-h_t)=&\log(h+\lambda(\varphi|\nabla u|^q)^{-1}\mathcal{K}hf-h)\\
\nonumber=&\log\mathcal{K}+
\log(\lambda(\varphi|\nabla u|^q)^{-1}hf)\\
=&-\log(\det(\overline{\nabla}^2h+hI))+\log(\lambda(\varphi|\nabla u|^q)^{-1}hf).
\end{align}
Let
\begin{align*}
\chi(x,t)=\log(\lambda(\varphi|\nabla u|^q)^{-1}hf).
\end{align*}
Differentiating (\ref{eq416}) once and twice, we respectively get
\begin{align*}\frac{h_k-h_{kt}}{h-h_t}=&-\sum b^{ij}\nabla_kb_{ij}+\nabla_k\chi\\
	=&-\sum b^{ii}(h_{kii}+h_i\delta_{ik})+\nabla_k\chi,
\end{align*}
and
\begin{align*}
\frac{h_{11}-h_{11t}}{h-h_t}-\frac{(h_1-h_{1t})^2}{(h-h_t)^2}=&-\bigg(-\sum (b^{ii})^2(\nabla_{i}b_{ii})^2+b^{ii}\nabla_{ii}b_{ii}\bigg)+\nabla_{11}\chi\\
=&-\sum b^{ii}\nabla_{11}b_{ii}+\sum b^{ii}b^{jj}(\nabla_1b_{ij})^2+\nabla_{11}\chi.
\end{align*}
By the Ricci identity, we have
\begin{align*}
\nabla_{11}b_{ii}=\nabla_{ii}b_{11}-b_{11}+b_{ii}.
\end{align*}
Thus, we can derive
\begin{align*}
\frac{\frac{\partial}{\partial t}\widetilde{\digamma}}{h-h_t}=&b^{11}\bigg(\frac{h_{11t}+h_t}
	{h-h_t}\bigg)-A\frac{h_t}{h(h-h_t)}
	+\frac{2B\sum h_kh_{kt}}{h-h_t}\\
	=&b^{11}\bigg(\frac{h_{11t}-h_{11}}{h-h_t}+\frac{h_{11}+h-h+h_t}{h-h_t}\bigg)-A\frac{1}{h}
	\frac{h_t-h+h}{h-h_t}
	+\frac{2B\sum h_kh_{kt}}{h-h_t}\\
	=&b^{11}\bigg(-\frac{(h_1-h_{1t})^2}{(h-h_t)^2}+\sum b^{ii}
	\nabla_{11}b_{ii}-\sum b^{ii}b^{jj}(\nabla_1b_{ij})^2-\nabla_{11}\chi\bigg.\\
	&\bigg.+\frac{h_{11}+h-{(h-h_t)}}
	{h-h_t}\bigg) -\frac{A}{h}\bigg(\frac{-(h-h_t)+h}{h-h_t}\bigg)+\frac{2B\sum h_kh_{kt}}{h-h_t}\\
	=&b^{11}\bigg(-\frac{(h_1-h_{1t})^2}{(h-h_t)^2}+\sum b^{ii}
	\nabla_{11}b_{ii}-\sum b^{ii}b^{jj}(\nabla_1b_{ij})^2-\nabla_{11}\chi\bigg)\\
	&+b^{11}\bigg(\frac{h_{11}+h}{h-h_t}-1\bigg) +\frac{A}{h}-\frac{A}{h-h_t}+\frac{2B\sum h_kh_{kt}}{h-h_t}\\
	=&b^{11}\bigg(-\frac{(h_1-h_{1t})^2}{(h-h_t)^2}+\sum b^{ii}
	\nabla_{11}b_{ii}-\sum b^{ii}b^{jj}(\nabla_1b_{ij})^2-\nabla_{11}\chi\bigg)+\frac{1-A}{h-h_t}\\
	&-b^{11}+\frac{A}{h}+\frac{2B\sum h_kh_{kt}}{h-h_t}\\
    \leq&b^{11}\bigg(\sum b^{ii}(\nabla_{ii}b_{11}-b_{11}+b_{ii})-\sum b^{ii}b^{jj}(\nabla_1b_{ij})^2\bigg)
	-b^{11}\nabla_{11}\chi+\frac{1-A}{h-h_t}\\
	&+\frac{A}{h}+\frac{2B\sum h_kh_{kt}}{h-h_t}\\
	\leq&\sum b^{ii}\bigg[(b^{11})^2(\nabla_ib_{11})^2+A\bigg(\frac{h_{ii}}{h}-\frac{h_i^2}{h^2}
	\bigg)-2B(\sum h_kh_{kii}+h_{ii}^2)\bigg]\\
	&-b^{11}\sum b^{ii}b^{jj}(\nabla_1b_{ij})^2-b^{11}\nabla_{11}\chi+\frac{1-A}{h-h_t}+\frac{A}{h}+\frac{2B\sum h_kh_{kt}}{h-h_t}\\
	\leq&\sum b^{ii}\bigg[A\bigg(\frac{h_{ii}+h-h}{h}-\frac{h_i^2}{h^2}\bigg)\bigg]+2B
	\sum h_k\bigg(-\sum b^{ii}h_{kii}+\frac{h_{kt}}{h-h_t}\bigg)\\
	&-2B\sum b^{ii}(b_{ii}-h)^2-b^{11}\nabla_{11}\chi+\frac{1-A}{h-h_t}+\frac{A}{h}\\
	\leq&\sum b^{ii}\bigg[A\bigg(\frac{b_{ii}}{h}-1\bigg)\bigg]+2B\sum h_k\bigg(\frac{h_k}{h-h_t}
	+b^{kk}h_k-\nabla_k\chi\bigg)\\
	&-2B\sum b^{ii}(b_{ii}^2-2b_{ii}h)-b^{11}\nabla_{11}\chi+\frac{1-A}{h-h_t}+\frac{A}{h}\\
	\leq&-2B\sum h_k\nabla_k\chi-b^{11}\nabla_{11}\chi+(2B|\overline{\nabla} h|-A)\sum b^{ii}-2B\sum b_{ii}\\
	&+4B(n-1)h+\frac{2B|\overline{\nabla} h|^2+1-A}{h-h_t}+\frac{nA}{h}.
\end{align*}
Recall
\begin{align*}
\chi(x,t)=\log(\lambda(\varphi|\nabla u|^q)^{-1}hf)=\log \lambda-\log \varphi-q\log|\nabla u|+\log h+\log f,
\end{align*}
since $\lambda$ is a constant factor, we have $\lambda_k=0$.
Consequently, we may obtain following form by $\chi(x,t)$ and (\ref{eq415}),
\begin{align*}
&-2B\sum h_k\nabla_k\chi-b^{11}\nabla_{11}\chi\\
=&-2B\sum h_k\bigg(\frac{f_k}{f}+\frac{h_k}{h}-q\frac{(|\nabla u|)_k}{|\nabla u|}
-\frac{\varphi^{\prime}h_k}{\varphi}\bigg)-b^{11}\nabla_{11}\chi\\
=&-2B\sum h_k\bigg(\frac{f_k}{f}+\frac{h_k}{h}-q\frac{(|\nabla u|)_k}{|\nabla u|}
-\frac{\varphi^{\prime}h_k}{\varphi}\bigg)\\
&-b^{11}\bigg(\frac{ff_{11}-f_1^2}{f^2}+\frac{hh_{11}-h_1^2}{h^2}-q\frac{|\nabla u|(|\nabla u|)_{11}-(|\nabla u|)^2_1}{(|\nabla u|)^2}-\frac{\varphi^{\prime \prime}h_1^2+\varphi^{\prime}h_{11}}{\varphi}
+\frac{(\varphi^{\prime}h_1)^2}{\varphi^2}\bigg)\\
\leq&C_9B+C_{10}b^{11}+2qB\sum h_k\frac{(|\nabla u|)_k}{|\nabla u|}+b^{11}\frac{h(b_{11}-h)}{h^2}\\
&+qb^{11}\frac{|\nabla u|(|\nabla u|)_{11}-(|\nabla u|)^2_1}{(|\nabla u|)^2}+b^{11}\bigg(\frac{\varphi^{\prime \prime}h_1^2+\varphi^{\prime}h_{11}}{\varphi}
+\frac{(\varphi^{\prime}h_1)^2}{\varphi^2}\bigg),
\end{align*}
where $\varphi^{\prime\prime}=\frac{\partial^2\varphi(s)}{\partial s^2}$,
\begin{align*}
b^{11}\bigg(\frac{\varphi^{\prime \prime}h_1^2+\varphi^{\prime}h_{11}}{\varphi}
+\frac{(\varphi^{\prime}h_1)^2}{\varphi^2}\bigg)=b^{11}\bigg(\frac{\varphi^{\prime \prime}h_1^2+\varphi^{\prime}(b_{11}-h)}{\varphi}
+\frac{(\varphi^{\prime}h_1)^2}{\varphi^2}\bigg)\leq C_{11}b^{11}.
\end{align*}
Recall that
\begin{align*}|\nabla u(X,t)|=-\nabla u(X,t)\cdot x,
\end{align*}
taking the covariant derivative above equality, we get
\begin{align*}
(|\nabla u|)_k=-b_{ik}((\nabla^2u)e_i\cdot x),
\end{align*}
further,
\begin{align*}
(|\nabla u|)_{11}=&-b_{i11}((\nabla^2u)e_i\cdot x)-b_{j1}b_{i1}((\nabla^3 u)e_je_i\cdot x)\\
&+b_{i1}((\nabla^2u)x\cdot x)-b_{i1}((\nabla^2u)e_i\cdot e_1).
\end{align*}
Thus, combining (\ref{eq403}) with Lemma \ref{lem47}, we get
\begin{align*}
2qB\sum h_k\frac{(|\nabla u|)_k}{|\nabla u|}=2qB\sum h_k\frac{-b_{ik}((\nabla^2u)e_i\cdot x)}{|\nabla u|}\leq C_{12}Bb_{11}.
\end{align*}
From (\ref{eq415}), we obtain
\begin{align*}
b^{11}b_{i11}=A\frac{h_i}{h}+2Bh_ih_{ii}=A\frac{h_i}{h}+2Bh_i(b_{ii}-h\delta_{ii}),
\end{align*}
therefore, from (\ref{eq403}), (\ref{eq406}) and Lemma \ref{lem47}, we get
\begin{align*}
qb^{11}\frac{|\nabla u|(|\nabla u|)_{11}-(|\nabla u|)^2_1}{(|\nabla u|)^2}\leq C_{13}Bb_{11}.
\end{align*}
It follows that
\begin{align*}
\frac{\frac{\partial}{\partial t}\widetilde{\digamma}}{h-h_t}\leq C_{14}Bb_{11}+C_{15}b^{11}+
(2B|\overline{\nabla} h|-A)\sum b^{ii}-2B\sum b_{ii}+4B(n-1)h+\frac{nA}{h}<0,
\end{align*}
provided $b_{11}>>1$ and if we choose $A>>B$. We obtain
\begin{align*}
\widetilde{\digamma}(x_0,t)\leq C,
\end{align*}
hence,
\begin{align*}
\digamma(x_0,t)=\widetilde{\digamma}(x_0,t)\leq C.
\end{align*}
This tells us the principal radii are bounded from above, or equivalently $\kappa_i\geq\frac{1}{C}$.
\end{proof}

\vskip 0pt
\section{\bf The convergence of the flow}\label{sec5}

With the help of priori estimates in the section \ref{sec4}, the long-time existence and asymptotic behaviour of the flow (\ref{eq109}) (or (\ref{eq301})) are obtained, we also can complete proof of Theorem \ref{thm13}.
\begin{proof}[Proof of the Theorem \ref{thm13}] Since Eq. (\ref{eq303}) is parabolic, we can get its short time existence. Let $T$ be the maximal time such that $h(\cdot, t)$ is a smooth, non-even and strictly convex solution to Eq. (\ref{eq303}) for all $t\in[0,T)$. Lemma \ref{lem43}-\ref{lem45}, Remark \ref{rem46} and Lemma \ref{lem47} enable us to apply Lemma \ref{lem48} to Eq. (\ref{eq303}), thus, we can deduce a
uniformly upper and lower bounds for the biggest eigenvalue of $\{(h_{ij}+h\delta_{ij})(x,t)\}$. This implies
\begin{align*}
C^{-1}I\leq (h_{ij}+h\delta_{ij})(x,t)\leq CI,\quad \forall (x,t)\in S^{n-1}\times [0,T),
\end{align*}
where $C>0$ independents on $t$. This shows that Eq. (\ref{eq303}) is uniformly parabolic. Estimates for higher derivatives follows from the standard regularity theory of uniformly parabolic equations
Krylov \cite{KR}. Hence, we obtain the long time existence and regularity of solutions for the flow
(\ref{eq109}) (or (\ref{eq301})). Moreover, we obtain
\begin{align}\label{eq501}
\|h\|_{C^{l,m}_{x,t}(S^{n-1}\times [0,T))}\leq C_{l,m},
\end{align}
for some $C_{l,m}$ ($l, m$ are nonnegative integers pairs) independent of $t$, then $T=\infty$. Using parabolic comparison principle, we can attain the uniqueness of the smooth, non-even and strictly convex solution $h(\cdot,t)$ of Eq. (\ref{eq303}).

Now, recall the non-increasing property of $\Gamma(\Omega_t)$ in Lemma \ref{lem32}, we know that

\begin{align}\label{eq502}
\frac{\partial\Gamma(\Omega_t)}{\partial t}\leq 0.
\end{align}
Based on (\ref{eq502}), there exists a $t_0$ such that
\begin{align*}
\frac{\partial\Gamma(\Omega_t)}{\partial t}\bigg|_{t=t_0}=0,
\end{align*}
this yields
\begin{align*}
\tau\varphi(h)|\nabla u|^q\det(\nabla_{ij}h+h\delta_{ij})=f.
\end{align*}
Let $\Omega=\Omega_{t_0}$, thus, $\Omega$ satisfies (\ref{eq108}).

In view of (\ref{eq501}), applying the Arzel$\grave{a}$-Ascoli theorem \cite{BRE} and a diagonal argument, we can extract
a subsequence of $t$, denoted by $\{t_j\}_{j \in \mathbb{N}}\subset (0,+\infty)$, and there exists a smooth function $h(x)$ such that
\begin{align}\label{eq503}
\|h(x,t_j)-h(x)\|_{C^l(S^{n-1})}\rightarrow 0,
\end{align}
uniformly for each nonnegative integer $l$ as $t_j \rightarrow \infty$. This reveals that $h(x)$ is a support function. Let us denote by $\Omega$ the convex body determined by $h(x)$. Thus, $\Omega$ is smooth, strictly convex and containing the origin in its interior.

Moreover, by (\ref{eq501}) and the uniform estimates in section \ref{sec4}, we conclude that $\Gamma(\Omega_t)$ is a bounded function in $t$ and $\frac{\partial \Gamma(\Omega_t)}{\partial t}$ is uniformly continuous. Thus, for any $t>0$, by monotonicity of the $\Gamma(\Omega_t)$ in Lemma \ref{lem32}, there is a constant $C>0$ independent of $t$, such that
\begin{align*}
\int_0^t\bigg(-\frac{\partial\Gamma(\Omega_t)}{\partial t}\bigg)dt=\Gamma(\Omega_0)-\Gamma(\Omega_t)\leq C,
\end{align*}
this gives
\begin{align}\label{eq504}\lim_{t\rightarrow\infty}\Gamma(\Omega_t)-\Gamma(\Omega_0)=
-\int_0^\infty\bigg|\frac{\partial}{\partial t}\Gamma(\Omega_t)\bigg|dt\leq C.
\end{align}
The left hand side of (\ref{eq504}) is bounded below by $-2C$, therefore, there is a subsequence $t_j\rightarrow\infty$ such that
\begin{align*}
\frac{\partial}{\partial t}\Gamma(\Omega_{t_j})\rightarrow 0 \quad\text{as}\quad  t_j\rightarrow\infty.
\end{align*}
The proof of Lemma \ref{lem32} shows that
\begin{align*}
\frac{\partial\Gamma(\Omega_t)}{\partial t}\bigg|_{t=t_j}=-\frac{\int_{S^{n-1}}|\nabla u|^q\frac{h}
{\mathcal{K}}dx}{\int_{S^{n-1}}\frac{hf}{\varphi(h)}dx}\int_{S^{n-1}}\frac{f^{2}(x)h}{|\nabla u|^q\varphi^2(h)}\mathcal{K}dx+\int_{S^{n-1}}\frac{hf(x)}{\varphi(h)}dx\leq 0,
\end{align*}
taking the limit $t_j\rightarrow \infty$, by equality condition of H\"{o}lder inequality, it means that there exists
\begin{align*}
\tau\varphi(h^{\infty})|\nabla u(X^{\infty})|^q\det(\nabla_{ij}h^{\infty}+h^{\infty}\delta_{ij})=f(x),
\end{align*}
which satisfies (\ref{eq108}) with $\tau$ given by
\begin{align*}\frac{1}{\tau}=\lim_{t_j\rightarrow\infty}\lambda(t_j),
\end{align*}
where $h^{\infty}$ is the support function of $\Omega^{\infty}$, and $X^{\infty}=\nabla h^{\infty}$. This completes the proof of Theorem \ref{thm13}.
\end{proof}

\end{document}